\documentclass[leqno, 12pt]{article}
\usepackage{amsthm,amsmath,amsfonts,amssymb,indentfirst}
\usepackage{xy} \xyoption{all}

\newtheorem{theorem}{Theorem}
\newtheorem{lemma}[theorem]{Lemma}
\newtheorem{corollary}[theorem]{Corollary}
\newtheorem{proposition}[theorem]{Proposition}

\theoremstyle{definition}
\newtheorem{example}[theorem]{Example}

\newtheorem{definition}[theorem]{Definition}

\newcommand{\End}{\mathrm{End}}

\newcommand{\path}{\mathrm{Path}}
\newcommand{\clpath}{\mathrm{ClPath}}

\newcommand{\Z}{\mathbb{Z}}

\newcommand{\N}{\mathbb{N}}
\newcommand{\M}{\mathbb{M}}

\begin{document}

\title{Commutator Leavitt Path Algebras}
\author{Zachary Mesyan}

\maketitle

\begin{abstract}  For any field $K$ and directed graph $E$, we completely describe the elements of the Leavitt path algebra $L_K(E)$ which lie in the  commutator subspace $[L_K(E), L_K(E)]$. We then use this result to classify all Leavitt path algebras $L_K(E)$ that satisfy $L_K(E) = [L_K(E),L_K(E)]$. We also show that these Leavitt path algebras have the additional (unusual) property that all their Lie ideals are (ring-theoretic) ideals, and construct examples of such rings with various ideal structures.

\medskip

\noindent
\emph{Keywords:}  Leavitt path algebra, commutator, perfect Lie algebra 

\noindent
\emph{2010 MSC numbers:} 16W10, 16D70, 17B60.
\end{abstract}

\section{Introduction}

In 2005 Abrams and Aranda Pino~\cite{AP} introduced the \emph{Leavitt path algebras}, which generalize a class of algebras universal with respect to an isomorphism property between finite-rank free modules, studied by Leavitt~\cite{Leavitt}, and provide algebraic analogues of graph $C^*$-algebras (see~\cite{Raeburn}). Subsequently, Leavitt path algebras have become a lively research subject.

Let $K$ be a field, $E$ a directed graph, and $L_K(E)$ the associated Leavitt path algebra (described in detail in Definition~\ref{LPAdefinition}). In this paper we build on work began in~\cite{AF} and~\cite{AM}, investigating the structure of the Lie algebra, or commutator subspace, $[L_K(E), L_K(E)]$ (consisting all $K$-linear combinations of elements of the form $xy-yx$, for $x,y\in L_K(E)$) arising from $L_K(E)$. More specifically, in those two papers it is determined when $[L_K(E), L_K(E)]$ is simple as a Lie algebra, for certain simple Leavitt path algebras $L_K(E)$. Along the way, a necessary and sufficient condition is given in~\cite{AM} for an arbitrary linear combination of vertices of the graph $E$ to lie in $[L_K(E), L_K(E)]$.

Here we extend the above result, to produce a complete classification of the elements of $L_K(E)$ that lie in $[L_K(E), L_K(E)]$ (Theorem~\ref{All-Comm}), for arbitrary $K$ and $E$. Since any full $d \times d$ matrix ring $\M_d(K)$ over a field $K$ can be expressed as a Leavitt path algebra (as discussed in the next section), our result can also be viewed as a generalization of the classical theorem of Albert/Muckenhoupt~\cite{AABM} and Shoda~\cite{Shoda} showing that a matrix in $\M_d(K)$ is a sum of commutators if and only if it has trace zero. (For a generalization in a different direction see~\cite{ZM}.)

In Theorem~\ref{PerfectLPAs} we use the above classification to completely describe the Leavitt path algebras that satisfy $L_K(E) = [L_K(E), L_K(E)]$. We also give many examples, having different ideal structures, of Leavitt path algebras with this property. Rings $R$ satisfying $R=[R,R]$ (the additive subgroup of $R$ generated by the commutators), or \emph{commutator rings}, are rather rare beasts--the few previously known examples are discussed in~\cite{ZM}, and the aforementioned Leavitt path algebras significantly expand this class. The study of commutator rings dates back to at least 1956, when Kaplansky~\cite{Kaplansky} asked, along with eleven other ring-theoretic questions, whether there is a division ring in which every element is a sum of commutators. (This was answered in the affirmative two years later by Harris~\cite{Harris}.) Lie algebras $L$ (generally, not arising from associative rings) with the property $L=[L,L]$ have also been studied in the literature (e.g., see~\cite{Jacobson}, \cite{BZ}, and~\cite{Marshall}), and are known as \emph{perfect Lie algebras}.

It turns out that all the commutator Leavitt path algebras $L_K(E)$ have the further property that a subset $U \subseteq L_K(E)$ is an ideal if and only if it is a Lie ideal (i.e., a $K$-subspace $U$ of $L_K(E)$ satisfying $[U, L_K(E)] \subseteq U$), and in this case $[U, L_K(E)] = U$ (Theorem~\ref{ideal_iff_Lie_ideal}). The author is not aware of any other (necessarily nonunital) rings with this property in the literature. That the ideals and Lie ideals coincide in commutator Leavitt path algebras is particularly interesting, since in general it seems to be difficult to describe the Lie ideal structure of a ring. While this subject has been studied extensively (e.g., see~\cite{BHK}, \cite{Erickson}, \cite{Herstein1}, \cite{Herstein2}, and~\cite{LM}), a complete classification of the Lie ideals seems to be available only for simple rings.

In the final section of this paper we note that commutator Leavitt path algebras actually belong to a broader class of commutator rings which arise as direct limits. It is hard, however, to say much about this more general class. For more on Leavitt path algebras as direct limits see~\cite{Goodearl}.

\subsection*{Acknowledgement} 

The author is grateful to Gene Abrams for very helpful discussions about this subject.

\section{Definitions and Notation}\label{NotationSect}

We begin by recalling some notions from the Lie theory of associative rings, defining the Leavitt path algebra of a graph, and establishing notation that will be used throughout the paper.

Let $R$ be an associative (but not necessarily unital) ring. Given two elements $x,y \in R$, we let $[x, y]$ denote the commutator $xy - yx$, and let $[R, R]$ denote the additive subgroup of $R$ generated by the commutators. Then both $R$ and $[R,R]$ become (non-associative) Lie rings, with operation $x \ast y = [x,y]$. If $R$ is in addition an algebra over a field $K$, then $[R,R]$ is a $K$-subspace of $R$ (since $k[x,y] = [kx,y]$ for all $k \in K$ and $x,y \in R$), and so $R$ and $[R,R]$ are also Lie $K$-algebras. If $R = [R, R]$, then $R$ is said to be \emph{commutator ring}, or to be \emph{perfect} as a Lie ring. A subset $U$ of a Lie ring (respectively, Lie $K$-algebra) $L$ is called a \emph{Lie ideal} if $U$ is an additive subgroup (respectively, $K$-subspace) of $L$ and $[L, U] \subseteq U$ (that is, $[l,u] \in U$ for all $l \in L$ and $u \in U$). Also, given a set $I$, we let $R^{(I)}$ denote the direct sum of copies of $R$ indexed by $I$. The sets of all integers, natural numbers, and positive integers will be denoted by $\Z$, $\N$, and $\Z_+$, respectively.

\subsection{Graphs}

A \emph{directed graph} $E=(E^0,E^1,r,s)$ consists of two  sets $E^0,E^1$ (the elements of which are called \emph{vertices} and \emph{edges}, respectively), together with functions $s,r:E^1 \to E^0$, called \emph{source} and \emph{range}, respectively. The word \emph{graph} will always mean \emph{directed graph}. A \emph{path} $p$ in  $E$ is a finite sequence of (not necessarily distinct) edges $p=e_1\dots e_n$ such that $r(e_i)=s(e_{i+1})$ for $i=1,\dots,n-1$. Here we define $s(p):=s(e_1)$ to be the \emph{source} of $p$, $r(p):=r(e_n)$ to be the \emph{range} of $p$, and $|p| : = n$ to be the \emph{length} of $p$. We view the elements of $E^0$ as paths of length $0$, and denote by $\path(E)$ the set of all paths in $E$. A path $p = e_1\dots e_n \in \path(E)$ ($e_1, \dots, e_n \in E^1$) is said to be \emph{closed} if $s(p)=r(p)$. Such a path is said to be a \emph{cycle} if in addition $s(e_i)\neq s(e_j)$ for every $i\neq j$. A cycle consisting of just one edge is called a \emph{loop}. We denote the set of all closed paths in $\path(E)$ by $\clpath(E)$. A graph which contains no cycles is called \emph{acyclic}.

Given two vertices $u, v \in E^0$ such that there is a path $p \in \path(E)$ satisfying $s(p) = u$ and $r(p) = v$, we let $d(u,v)$ be the length of a shortest such path $p$. (In particular, we have $d(u,u) = 0$.) If no such path $p$ exists, we set $d(u,v) = \infty$.  Also, for each $u \in E^0$ and $m \in \N$ let $D(u, m) = \{v \in E^0 \mid d(u,v) \leq m\}$. A subset $H \subseteq E^0$ is said to be \emph{hereditary} if $v \in H$ whenever $u \in H$ and $d(u, v) < \infty$. Also, $H$ is \emph{saturated} if $v \in H$ whenever $0 < |s^{-1}(v)| < \infty$ and $\{r(e) \mid e \in E^1 \text{ and } s(e) = v\} \subseteq H$.

A vertex $v \in E^0$ for which the set $\{e\in E^1 \mid s(e)=v\}$ is finite is said to have \emph{finite out-degree}. The graph $E$ is said to have \emph{finite out-degree}, or to be \emph{row-finite}, if every vertex of $E$ has finite out-degree. A graph $E$ where both $E^0$ and $E^1$ are finite sets is called a \emph{finite} graph. A vertex $v \in E^0$ which is the source vertex of no edges of $E$ is called a \emph{sink}, while a vertex having finite out-degree which is not a sink is called a \emph{regular} vertex.

\begin{definition} \label{number-def}
Let $E$ be a graph, and for each $v \in E^0$ let $\epsilon_v \in \Z^{(E^0)}$ denote the element with $1$ in the $v$-th coordinate and zeros elsewhere. If $v \in E^0$ is a regular vertex, for all $u \in E^0$ let $a_{vu}$ denote the number of edges $e \in E^1$ such that $s(e) = v$ and $r(e) = u$. In this situation, define $$B_v = (a_{vu})_{u \in E^0} - \epsilon_v \in \Z^{(E^0)}.$$   On the other hand, let $$B_v = (0)_{u \in E^0} \in \Z^{(E^0)},$$ if $v$ is not a regular vertex.  \hfill $\Box$
\end{definition}

\subsection{Path Algebras}

\begin{definition}\label{LPAdefinition}  
Let $K$ be a field, and let $E$ be a  graph. The \emph{Cohn path $K$-algebra} $C_K(E)$ \emph{of $E$ with coefficients in $K$} is the $K$-algebra generated by the set $\{v\mid v\in E^0\} \cup \{e,e^*\mid e\in E^1\}$, which satisfies the following relations:

(V)  $vw = \delta_{v,w}v$ for all $v,w\in E^0$ \  (i.e., $\{v\mid v\in E^0\}$ is a set of orthogonal idempotents),

(E1) $s(e)e=er(e)=e$ for all $e\in E^1$,

(E2) $r(e)e^*=e^*s(e)=e^*$ for all $e\in E^1$,

(CK1) $e^*f=\delta _{e,f}r(e)$ for all $e,f\in E^1$.\\
Let $N$ denote the ideal of $C_K(E)$ generated by elements of the form $v~-~\sum_{\{e\in E^1 \mid s(e)=v\}} ee^*,$ where $v \in E^0$ is a regular vertex. Then a $K$-algebra of the form $C_K(E)/N$ is called a \emph{Leavitt path $K$-algebra} $L_K(E)$ \emph{of $E$ with coefficients in $K$}. Equivalently, $L_K(E)$ is the $K$-algebra generated by the set $\{v\mid v\in E^0\} \cup \{e,e^*\mid e\in E^1\}$, which satisfies the relations (V), (E1), (E2), and (CK1) above, together with the relation

(CK2) $v=\sum _{\{ e\in E^1\mid s(e)=v \}}ee^*$ for every regular vertex $v\in E^0$.   
\hfill $\Box$
\end{definition}

If $p = e_1 \dots e_n \in \path(E)$ ($e_1, \dots, e_n \in E^1$) is a path, then we set $p^*:=e_n^* \dots e_1^*$. An expression of this form is called a \emph{ghost path}. Given $p \in \path(E)$, we let $r(p^*)$ denote $s(p)$, and we let $s(p^*)$ denote $r(p)$. Setting $v^* := v$ for each $v \in E^0$, one can show that $$\{pq^* \mid p,q \in \path(E) \text{ such that } r(p)=r(q)\}$$ is a basis for $C_K(E)$ as a $K$-vector space, and hence this set also generates $L_K(E)$ as a $K$-vector space. It is possible to define the operation $()^*$ on all elements of $C_K(E)$ and $L_K(E)$ by setting $(pq^*)^* = qp^*$ for all $p, q \in \path(E)$, and then extending linearly. Thus, $C_K(E)$ and $L_K(E)$ can be equipped with an involution. (See~\cite{AP} for more details.) It is also easy to see that $C_K(E)$ and $L_K(E)$ are unital precisely when $E$ has finitely many vertices, in which case $1= \sum_{v \in E^0} v$.

Many well-known rings arise as Leavitt path algebras. For example, the classical Leavitt $K$-algebra $L_K(n)$ for $n\geq 2$ (see~\cite{Leavitt}) can be expressed as the Leavitt path algebra of the ``rose with $n$ petals" graph pictured below.
$$
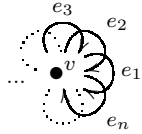
\xymatrix{ & {\bullet^v} \ar@(ur,dr) ^{e_1} \ar@(u,r) ^{e_2}
\ar@(ul,ur) ^{e_3} \ar@{.} @(l,u) \ar@{.} @(dr,dl) \ar@(r,d) ^{e_n}
\ar@{}[l] ^{\ldots} }$$
The full $d\times d$ matrix algebra $\M_d(K)$ over a field $K$ is isomorphic to the Leavitt path algebra of the oriented line graph having $d$ vertices, shown below.
$$\xymatrix{{\bullet}^{v_1} \ar [r] ^{e_1} & {\bullet}^{v_2}  \ar@{.}[r] & {\bullet}^{v_{d-1}} \ar [r]^{e_{d-1}} & {\bullet}^{v_d}}$$
Also, the Laurent polynomial algebra $K[x,x^{-1}]$ can be identified with the Leavitt path algebra of the following ``one vertex, one loop" graph.
$$
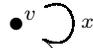
\xymatrix{{\bullet}^{v} \ar@(ur,dr) ^x}$$

\section{Arbitrary Commutators}

In this section we shall describe all the elements of the commutator subspace of an arbitrary Leavitt path algebra. First let us recall some results from~\cite{AM}.

\begin{lemma}[Lemma 9 in~\cite{AM}]\label{annihilators}
Let $v\in E^0$ be a regular vertex, and let $y$ denote the element
$ v - \sum_{\{e\in E^1 \mid s(e)=v\}}ee^*$  of the ideal $N$ of $C_K(E)$, described in Definition \ref{LPAdefinition}. Then for all $p \in \path(E) \setminus E^0$ we have $yp=0=p^*y$.
\end{lemma}

\begin{definition}\label{TDefinition}
Let $K$ be a field, and let $E$ be a graph. For each $v \in E^0$, let $\epsilon_v \in K^{(E^0)}$ denote the element with $1 \in K$ in the $v$-th coordinate and zeros elsewhere. Let $T : C_K(E) \rightarrow K^{(E^0)}$ be the $K$-linear map which acts as $$T(pq^*) = \left\{ \begin{array}{cl}
\epsilon_v & \textrm{if } q^*p = v\\
0 & \textrm{otherwise}
\end{array}\right.$$ 
on the aforementioned basis of $C_K(E)$. \hfill $\Box$
\end{definition}

\begin{lemma}[Lemmas 11 and 13 in~\cite{AM}] \label{trace2}
Let $K$ be a field, and let $E$ be  graph. Then the following hold. 
\begin{enumerate}
\item[$(1)$] For all $x,y \in C_K(E)$ we have $T(xy) = T(yx)$.
\item[$(2)$] For all $w \in N \subseteq C_K(E)$ we have $T(w) \in {\rm span}_K \{B_v \mid v \in E^0\} \subseteq K^{(E^0)}$.
\end{enumerate}
\end{lemma}

\begin{theorem}[Theorem 14 in~\cite{AM}] \label{LPA-comm}
Let $K$ be a field, let $E$ be a graph, and let $\, \{k_v \mid v\in E^0\} \subseteq  K$ be a set of scalars where  $k_v = 0$ for all but finitely many $v \in E^0$.  Then 
$$\sum_{v\in E^0}k_vv \in [L_K(E), L_K(E)] \mbox{ if and only if } \, (k_v)_{v\in E^0} \in \mathrm{span}_K\{B_v \mid v \in E^0\}.$$ 
\end{theorem}

\begin{lemma}[Lemma 15 in~\cite{AM}]\label{commutatorlemma}
Let $K$ be a field, $E$ a  graph, and $p,q \in \path(E) \setminus E^0$ any paths.
\begin{enumerate}
\item[$(1)$] If $s(p) \neq r(p)$, then $p,p^* \in  [L_K(E), L_K(E)]$.
\item[$(2)$] If $p\neq qx$ and $q \neq px$ for all $x \in \clpath(E)$, then $pq^* \in  [L_K(E), L_K(E)]$.
\end{enumerate}
\end{lemma}



Our goal for the rest of the section is to generalize Theorem~\ref{LPA-comm} by describing all the elements in an arbitrary Leavitt path algebra that can be expressed as sums of commutators. The argument proceeds through five lemmas.

\begin{lemma} \label{commsums1}
Let $K$ be a field, $E$ a  graph, $k_1, \dots, k_n \in K$, and $p_1, \dots, p_n \in {\rm Path}(E)$. Then $$\sum_{i=1}^n k_ip_ip_i^* \in [L_K(E), L_K(E)] \ \text{if and only if} \ \sum_{i=1}^n k_ir(p_i) \in [L_K(E), L_K(E)].$$
\end{lemma}

\begin{proof}
We have $$\sum_{i=1}^n k_ip_ip_i^* = \sum_{i=1}^n k_i([p_i,p_i^*] + r(p_i)) = \sum_{i=1}^n k_ir(p_i) + \sum_{i=1}^n k_i[p_i,p_i^*],$$ from which the lemma follows.
\end{proof}

\begin{lemma} \label{commsums2}
Let $K$ be a field, $E$ a  graph, $a_1, \dots, a_n, b_1, \dots, b_m \in K$, and $$p_1, \dots, p_n, q_1, \dots, q_m, t_1, \dots, t_m \in {\rm Path}(E),$$ where $q_i \neq t_i$ for all $i \in \{1, \dots, m\}$. Then $$\sum_{i=1}^n a_ip_ip_i^* + \sum_{i=1}^m b_iq_it_i^*\in [L_K(E), L_K(E)]$$ if and only if $$\sum_{i=1}^n a_ip_ip_i^* \in [L_K(E), L_K(E)] \ \text{and} \ \sum_{i=1}^m b_iq_it_i^*\in [L_K(E), L_K(E)].$$
\end{lemma}

\begin{proof}
The ``if" direction is clear. For the converse, viewing $L_K(E)$ as $C_K(E)/N$, we shall show that if $$\sum_{i=1}^n a_ip_ip_i^* + \sum_{i=1}^m b_iq_it_i^* + N \in [L_K(E), L_K(E)]$$ for some $a_i, b_i \in K$ and $p_i, q_i, t_i \in {\rm Path}(E)$ satisfying $q_i \neq t_i$ for all $i \in \{1, \dots, m\}$, then $$\sum_{i=1}^n a_ip_ip_i^* + N  \in [L_K(E), L_K(E)] \ \text{and} \ \sum_{i=1}^m b_iq_it_i^* + N  \in [L_K(E), L_K(E)].$$ Now, if $$\sum_{i=1}^n a_ip_ip_i^* + \sum_{i=1}^m b_iq_it_i^* + N \in [L_K(E), L_K(E)],$$ then there are elements $x_j,y_j \in C_K(E)$ and $w \in N$ such that $$\sum_{i=1}^n a_ip_ip_i^* +   \sum_{i=1}^m b_iq_it_i^* = \sum_j [x_j,y_j]  + w.$$ By Definition~\ref{TDefinition} and Lemma~\ref{trace2}(1), we have $$\sum_{i=1}^n a_iT(r(p_i)) = T(\sum_{i=1}^n a_ip_ip_i^* +   \sum_{i=1}^m b_iq_it_i^*) =T( \sum_j [x_j,y_j] ) + T(w) = 0 + T(w) = T(w).$$
By Theorem~\ref{LPA-comm} and Lemma~\ref{trace2}(2), this implies that $$\sum_{i=1}^n a_ir(p_i) + N \in [L_K(E), L_K(E)].$$ Hence, by Lemma~\ref{commsums1}, $$\sum_{i=1}^n a_ip_ip_i^* + N \in [L_K(E), L_K(E)],$$ and therefore also $$\sum_{i=1}^m b_iq_it_i^* + N  \in [L_K(E), L_K(E)],$$ as desired.
\end{proof}

\begin{lemma} \label{commsums3}
Let $K$ be a field, $E$ a  graph, and $a_1, \dots, a_n, b_1, \dots, b_m \in K$. Also let $$p_1, \dots, p_n, x_1, \dots, x_n, q_1, \dots, q_m, t_1, \dots, t_m \in {\rm Path}(E)$$ be such that $p_ix_i \neq 0 \neq q_it_i^*$ for all $i$. Then $$\sum_{i=1}^n a_ip_ix_ip_i^* + \sum_{i=1}^m b_iq_it_i^*q_i^*\in [L_K(E), L_K(E)]$$ if and only if $$\sum_{i=1}^n a_ix_i + \sum_{i=1}^m b_it_i^*\in [L_K(E), L_K(E)].$$
\end{lemma}

\begin{proof}
We have $$\sum_{i=1}^n a_ip_ix_ip_i^* + \sum_{i=1}^m b_iq_it_i^*q_i^* = \sum_{i=1}^n a_i ([p_ix_i,p_i^*] + x_i) + \sum_{i=1}^m b_i ([q_it_i^*,q_i^*] + t_i^*)$$ $$=\sum_{i=1}^n a_ix_i + \sum_{i=1}^m b_it_i^* + \sum_{i=1}^n a_i [p_ix_i,p_i^*] + \sum_{i=1}^m b_i [q_it_i^*,q_i^*],$$ from which the lemma follows.
\end{proof}

\begin{definition}
Let $E$ be a graph and $p, q \in \clpath(E)$. We write $p \sim q$ if there exist paths $x, y \in \path(E)$ such that $p=xy$ and $q=yx$. \hfill $\Box$
\end{definition}

\begin{lemma}
For any graph $E$, $\sim$ is an equivalence relation on the elements of $\,\clpath(E)$.
\end{lemma}

\begin{proof}
Let $p, q \in \clpath(E)$. Then clearly, $p \sim q$ if and only if $q \sim p$. Also $p \sim p$, since $r(p)=s(p)$, and hence $p = pr(p) = r(p)p$. Now, let $p,q,t \in \clpath(E)$ be such that $p \sim q$ and $q \sim t$. Then there are $w, x, y, z \in \path (E)$ such that $p = wx$, $q = xw = yz$, and $t = zy$. Since $xw=yz$, there must be some $u \in \path(E)$ such that either $y = xu$ or $x=yu$. In the first case, $w = uz$, and so $p=u(zx)$ and $t=(zx)u$. While in the second case, $z=uw$, and so $p=(wy)u$ and $t=u(wy)$. Either way, $p \sim t$.
\end{proof}

Given an element $t \in \clpath(E)$, we shall denote by $[t]$ the $\sim$-equivalence class of $t$. Note that if $t \in E^0$, then $[t] = \{t\}$.

Let us next define analogues of the trace function $T$ from Definition~\ref{TDefinition} for real and ghost closed paths that are not vertices. We shall use them to describe when linear combinations of such paths can be expressed as sums of commutators.

\begin{definition}\label{T2definition}
Let $K$ be a field, let $E$ be a graph, and denote by $S$ the set of all $\sim$-equivalence classes of elements of $\clpath(E)\setminus E^0$. For each $s \in S$, let $\epsilon_s \in K^{(S)}$ denote the element with $1 \in K$ in the $s$-th coordinate and zeros elsewhere.   Let $T_S : C_K(E) \rightarrow K^{(S)}$ be the $K$-linear map which acts as $$T_S(pq^*) = \left\{ \begin{array}{cl}
\epsilon_{[t]} & \text{if } q^*p = t \text{ for some } t\in \clpath(E)\setminus E^0\\
0 & \textrm{otherwise}
\end{array}\right.$$ 
on the basis of $C_K(E)$ described in Section~\ref{NotationSect}. Also, let $T_{S^*} : C_K(E) \rightarrow K^{(S)}$ be the $K$-linear map which acts as $$T_{S^*}(pq^*) = \left\{ \begin{array}{cl}
\epsilon_{[t]} & \text{if } q^*p=t^* \text{ for some } t\in \clpath(E)\setminus E^0\\
0 & \textrm{otherwise}
\end{array}\right.$$ 
on the same basis.\hfill $\Box$
\end{definition}

We note that for all $p, q \in \path(E)$, we have $T_S(pq^*) = T_{S^*}(qp^*)$, and hence $T_S(\mu) = T_{S^*}(\mu^*)$ for all $\mu \in L_K(E)$.

The next lemma is an analogue of Lemma~\ref{trace2} for $T_S$ and $T_{S^*}$.

\begin{lemma} \label{T2lemma}
Let $K$ be a field, and let $E$ be  graph. Then the following hold.
\begin{enumerate}
\item[$(1)$] For all $x,y \in C_K(E)$ we have $T_S(xy) = T_S(yx)$ and  $T_{S^*}(xy) = T_{S^*}(yx)$.  
\item[$(2)$] For all $w \in N \subseteq C_K(E)$ we have $T_S(w) = 0 = T_{S^*}(w)$.
\end{enumerate}
\end{lemma}

\begin{proof}
(1) Since $T_S$ and $T_{S^*}$ are $K$-linear, it is enough to establish the claim for $x$ and $y$ that are elements of the basis for $C_K(E)$ described above. That is, we may assume that $x=pq^*$ and $y=tz^*$, for some $p,q,t,z \in \path(E)$. Further, by the above observation, it is enough to prove the statement for $T_S$, since $T_{S^*}(xy) = T_S(y^*x^*)$ and $T_S(x^*y^*) = T_{S^*}(yx)$.

Now, suppose that $T_S(pq^*tz^*) \neq 0$. Then $pq^*tz^* = ghg^*$ for some $g \in \path(E)$ and $h \in \clpath(E) \setminus E^0$. Thus, there must be some $u, v \in \path(E)$ such that either $t=qu$, $p=zv$, $h=vu$ (and $z=g$), or $q=tu$, $p=zuh$ (and $g=zu$). In the first case, we have $$T_S(pq^*tz^*) = T_S(zvq^*quz^*) = T_S(zvuz^*) = \epsilon_{[vu]}$$ $$= \epsilon_{[uv]} = T_S(quvq^*) = T_S(quz^*zvq^*) = T_S(tz^*pq^*).$$ Similarly, in the second case we have $$T_S(pq^*tz^*) = T_S(zuhu^*t^*tz^*) = T_S(zuhu^*z^*) = \epsilon_{[h]}$$ $$= T_S(tuhu^*t^*) = T_S(tz^*zuhu^*t^*) = T_S(tz^*pq^*).$$
By a similar argument, it also follows that if $T_S(tz^*pq^*) \neq 0$, then $T_S(pq^*tz^*) = T_S(tz^*pq^*)$. Hence, $T_S(pq^*tz^*) = T_S(tz^*pq^*)$ for all values of $T_S(pq^*tz^*)$ and $T_S(tz^*pq^*)$.

(2) Again, using the above observation (and the fact that $N$ is closed under the operation $()^*$), along with the $K$-linearity of $T_S$, it is sufficient to show that for any generator $$y = v-\sum_{\{e\in E^1 \mid s(e)=v\}} ee^*$$ of $N$ and any two elements $c,c'$ of  $C_K(E)$, we have $T_S(cy_ic') = 0$. But, by (1), $T_S(cy_ic') = T_S(c'cy_i)$, and hence we only need to show that $T_S(cy_i) = 0$ for any $c \in C_K(E)$. Further, since $T_S$ is $K$-linear, we may assume that $c=pq^*$ belongs to the basis for $C_K(E)$ described above. Again, by (1), we have $T_S(cy)=T_S(pq^*y)=T_S(q^*yp)$.  But, by Lemma~\ref{annihilators}, the expression $q^*yp$ is zero unless $q^* = v = p$.  So it is enough to show that $T_S(y) = 0$. But clearly, $$T_S(y) = T_S( v - \sum_{\{e\in E^1 \mid s(e)=v\}} ee^*) = T_S(v) - \sum_{\{e\in E^1 \mid s(e)=v\}} T_S(ee^*)=0,$$ concluding the proof.
\end{proof}

We are now ready for our main result about commutators in an arbitrary Leavitt path algebra.

\begin{theorem} \label{All-Comm}
Let $K$ be a field, let $E$ be a graph, and let $\mu \in L_K(E)$ be an arbitrary element. Write $$\mu = \sum_{i=1}^k a_i p_ip_i^* + \sum_{i=1}^l b_i q_it_i^* + \sum_{i=1}^m\sum_{j=1}^{m_i} c_{ij} x_{ij}y_{ij}x_{ij}^* + \sum_{i=1}^n\sum_{j=1}^{n_i} d_{ij}z_{ij}w_{ij}^*z_{ij}^*$$ for some $k,l,m,m_i,n,n_i \in \Z_+$, $a_i, b_i, c_{ij}, d_{ij} \in K$, $p_i,q_i,t_i,x_{ij},z_{ij} \in \path(E)$, and $y_{ij}, w_{ij} \in \clpath(E)\setminus E^0$, where for each $i$ and $u \in \clpath(E)$, $q_i \neq t_iu$ and $t_i \neq q_iu$, and where $y_{ij} \sim y_{i'j'}$ if and only if $i=i'$ if and only if $w_{ij} \sim w_{i'j'}$. Also, for each $v \in E^0$ let $B_v$ and $\epsilon_v$ be as in Definition~\ref{number-def}.  

Then  $\mu \in [L_K(E), L_K(E)]$ if and only if the following conditions are satisfied:
\begin{enumerate}
\item[$(1)$] $\sum_{i=1}^k a_i\epsilon_{r(p_i)} \in \mathrm{span}_K \{B_v \mid v \in E^0 \},$
\item[$(2)$] $\sum_{j=1}^{m_i} c_{ij} = 0$ for each $i \in \{1, \dots, m\}$, 
\item[$(3)$] $\sum_{j=1}^{n_i} d_{ij} = 0$ for each $i \in \{1, \dots, n\}$.
\end{enumerate}
\end{theorem}

\begin{proof}
By Lemma~\ref{commutatorlemma}, $\sum_{i=1}^l b_i q_it_i^* \in [L_K(E), L_K(E)]$, and so Lemma~\ref{commsums2} implies that $\mu \in [L_K(E), L_K(E)]$ if and only if $$\sum_{i=1}^k a_i p_ip_i^* \in [L_K(E), L_K(E)] \text{ and } \sum_{i=1}^m\sum_{j=1}^{m_i} c_{ij} x_{ij}y_{ij}x_{ij}^* + \sum_{i=1}^n\sum_{j=1}^{n_i} d_{ij}z_{ij}w_{ij}^*z_{ij}^* \in [L_K(E), L_K(E)].$$ By Lemma~\ref{commsums1} and Theorem~\ref{LPA-comm}, $\sum_{i=1}^k a_i p_ip_i^* \in [L_K(E), L_K(E)]$ if and only if (1) holds. Hence, by Lemma~\ref{commsums3}, it is enough to show that $$\sum_{i=1}^m\sum_{j=1}^{m_i} c_{ij} y_{ij} + \sum_{i=1}^n\sum_{j=1}^{n_i} d_{ij}w_{ij}^* \in [L_K(E), L_K(E)]$$ if and only if (2) and (3) hold.

Suppose that (2) and (3) hold. We may then assume that each $m_i >1$ and each $n_i >1$. Since for each $i \in \{1, \dots, m\}$ and $j \in \{1, \dots, m_i-1\}$, $y_{ij} \sim y_{im_i}$, we can find $u_{ij}, v_{ij} \in \path(E)$ be such that $y_{ij}=u_{ij}v_{ij}$ and $y_{im_i} = v_{ij}u_{ij}$. Then, by (2), $$\sum_{i=1}^m\sum_{j=1}^{m_i} c_{ij} y_{ij} = \sum_{i=1}^m \bigg(\sum_{j=1}^{m_i-1} c_{ij} y_{ij} - \sum_{j=1}^{m_i-1} c_{ij}y_{im_i}\bigg) = \sum_{i=1}^m \sum_{j=1}^{m_i-1} c_{ij}(y_{ij} - y_{im_i})$$ $$= \sum_{i=1}^m \sum_{j=1}^{m_i-1} c_{ij}(u_{ij}v_{ij} - v_{ij}u_{ij}) \in [L_K(E), L_K(E)].$$ Similarly, for each $i \in \{1, \dots, n\}$ and $j \in \{1, \dots, n_i-1\}$, letting $g_{ij}, h_{ij} \in \path(E)$ be such that $w_{ij}=g_{ij}h_{ij}$ and $w_{in_i} = h_{ij}g_{ij}$, we have, by (3), $$\sum_{i=1}^n\sum_{j=1}^{n_i} d_{ij}w_{ij}^* = \sum_{i=1}^n \sum_{j=1}^{n_i-1} d_{ij}(h_{ij}^*g_{ij}^* - g_{ij}^*h_{ij}^*) \in [L_K(E), L_K(E)].$$ Hence, $$\sum_{i=1}^m\sum_{j=1}^{m_i} c_{ij} y_{ij} + \sum_{i=1}^n\sum_{j=1}^{n_i} d_{ij}w_{ij}^* \in [L_K(E), L_K(E)],$$ as desired.

For the converse, viewing $L_K(E)$ as $C_K(E)/N$, we shall show that if $$\sum_{i=1}^m\sum_{j=1}^{m_i} c_{ij} y_{ij} + \sum_{i=1}^n\sum_{j=1}^{n_i} d_{ij}w_{ij}^* + N \in [L_K(E), L_K(E)]$$ for some $m,m_i,n,n_i \in \Z_+$, $c_{ij}, d_{ij} \in K$, and $y_{ij}, w_{ij} \in \clpath(E)\setminus E^0$ as in the statement, then (2) and (3) hold. Write $$\sum_{i=1}^m\sum_{j=1}^{m_i} c_{ij} y_{ij} + \sum_{i=1}^n\sum_{j=1}^{n_i} d_{ij}w_{ij}^* = \sum_{i=1}^{k'}[g_i,h_i] + v$$ for some $k' \in \Z_+$, $g_i,h_i \in C_K(E)$, and $v \in N$. Then, by Lemma~\ref{T2lemma}, $$0=T_S\Big(\sum_{i=1}^m\sum_{j=1}^{m_i} c_{ij} y_{ij} + \sum_{i=1}^n\sum_{j=1}^{n_i} d_{ij}w_{ij}^*\Big) = \sum_{i=1}^m\sum_{j=1}^{m_i} c_{ij} T_S(y_{ij}) = \sum_{i=1}^m\sum_{j=1}^{m_i} c_{ij}\epsilon_{[y_{i1}]},$$ from which (2) follows. Similarly, $$0=T_{S^*}\Big(\sum_{i=1}^m\sum_{j=1}^{m_i} c_{ij} y_{ij} + \sum_{i=1}^n\sum_{j=1}^{n_i} d_{ij}w_{ij}^*\Big) = \sum_{i=1}^n\sum_{j=1}^{n_i} d_{ij} T_{S^*}(w_{ij}^*) = \sum_{i=1}^n\sum_{j=1}^{n_i} d_{ij}\epsilon_{[w_{i1}]}$$ implies (3).
\end{proof}

\begin{corollary} \label{CycleCor}
Let $K$ be a field, let $E$ be a graph, and let $\mu \in \path(E) \setminus E^0$ be a cycle. Then $\mu \notin [L_K(E), L_K(E)]$.
\end{corollary}

\begin{proof}
Using the notation of Theorem~\ref{All-Comm}, we have $a_i = 0$, $b_i = 0$, and $d_{ij} = 0$ for all $i$ and $j$. Moreover, exactly one of the $c_{ij}$ is equal to $1$, and the others are zero. Hence, condition (2) in the theorem is not satisfied, and therefore $\mu \notin [L_K(E), L_K(E)]$.
\end{proof}

\section{Commutator Rings}

As discussed in the Introduction, few examples of commutator rings are known (i.e., rings $R$ satisfying $R=[R,R]$). The goal of this section is to construct new such examples using Leavitt path algebras. The following consequence of Theorem~\ref{All-Comm} will help us identify the Leavitt path algebras having this property.

\begin{proposition} \label{comm-ring}
Let $K$ be a field, and let $E$ be a graph. Also, for each $v \in E^0$ let $B_v$ and $\epsilon_v$ be as in Definition~\ref{number-def}. Then the following are equivalent.
\begin{enumerate}
\item[$(1)$] $L_K(E) = [L_K(E),L_K(E)]$.
\item[$(2)$] $E$ is acyclic, and for all $u \in E^0$ we have $u \in [L_K(E), L_K(E)]$.
\item[$(3)$] $E$ is acyclic, and for all $u \in E^0$ we have $\epsilon_u \in \mathrm{span}_K \{B_v \mid v \in E^0 \}.$
\end{enumerate}
\end{proposition}

\begin{proof}
Corollary~\ref{CycleCor} implies that if $L_K(E)$ is a commutator ring, then $E$ must be acyclic.

Now suppose that $L_K(E)$ is acyclic. Then any element $\mu \in L_K(E)$ can be written in the form $$\mu = \sum_{i=1}^k a_i p_ip_i^* + \sum_{i=1}^l b_i q_it_i^*$$ for some $k,l \in \Z_+$, $a_i, b_i \in K$, and $p_i,q_i,t_i \in \path(E)$, where for each $i$ and $u \in \clpath(E) \ (=E^0)$, $q_i \neq t_iu$ and $t_i \neq q_iu$. Hence, by Theorem~\ref{All-Comm}, $\mu \in [L_K(E), L_K(E)]$ if and only if $\sum_{i=1}^k a_i\epsilon_{r(p_i)} \in \mathrm{span}_K \{B_v \mid v \in E^0 \}$. It follows that $L_K(E) = [L_K(E), L_K(E)]$ if and only if for all $u \in E^0$ we have $\epsilon_u \in \mathrm{span}_K \{B_v \mid v \in E^0 \},$ showing that (1) and (3) are equivalent.

The equivalence of (2) and (3) follows from Theorem~\ref{LPA-comm} (or Theorem~\ref{All-Comm}).
\end{proof}

The following lemma will be useful throughout the rest of the paper. (See Section~\ref{NotationSect} for the relevant notation.)

\begin{lemma} \label{path_sum}
Let $E$ be an acyclic graph, and let $u \in E^0$.
\begin{enumerate}
\item[$(1)$] Suppose that $E$ is finite, assume that $u$ is not a sink, and let $q_1, \dots, q_l \in \path(E)$ be all the paths such that $s(q_i) = u$ and $r(q_i)$ is a sink. Then $u = q_1q_1^*+ \dots + q_lq_l^*$.
\item[$(2)$] Suppose that $E$ has only regular vertices, let $m \in \N$, and let $q_1, \dots, q_l \in \path(E)$ be all the paths $q_i = e_1 \dots e_k$ $(e_1, \dots, e_k \in E^1)$ such that $s(q_i) = u$, $d(u,r(q_i)) = m+1$, and $s(e_2), \dots, s(e_k) \in D(u,m)$. Then $u = q_1q_1^*+ \dots + q_lq_l^*$.
\end{enumerate}
\end{lemma}

\begin{proof}
(1) We proceed by induction on $\sigma = |q_1| + \dots + |q_l|$. If $\sigma = 1$, then $l = 1$, and $q_1 \in E^1$ is the only edge that has source $u$. Thus $u = q_1q_1^*$, by the (CK2) relation (Definition~\ref{LPAdefinition}). Let us therefore assume that the statement holds for all $u \in E$ such that $\sigma = |q_1| + \dots + |q_l| \leq a$ ($a \geq 1$), and prove it for $u$ having $\sigma=a+1$. 

Let $v_1, \dots, v_n \in E^0$ be all the vertices such that $d(u, v_i)=1$, where $v_1, \dots, v_{m-1}$ are sinks, and $v_m, \dots, v_n$ are not sinks (for some $1\leq m \leq n$). Also, for each $i \in \{1, \dots, n\}$ let $e_{i,1}, \dots, e_{i,n_i} \in E^1$ be all the edges having source $u$ and range $v_i$, and for each $i \in \{m, \dots, n\}$ let $p_{i,1}, \dots, p_{i,m_i} \in \path(E)$ be all the paths having source $v_i$ and range that is a sink. We have $$\sum_{i=1}^l q_iq_i^* = \sum_{i=1}^{m-1} \sum_{j=1}^{n_i} e_{i,j}e_{i,j}^* + \sum_{i=m}^n \sum_{j=1}^{n_i} \sum_{k=1}^{m_i} e_{i,j}p_{i,k}p_{i,k}^*e_{i,j}^*.$$

Since $E$ is acyclic, for each $i \in \{m, \dots, n\}$ we have $$|p_{i,1}| + \dots + |p_{i,m_i}| < |q_1| + \dots + |q_l| = a+1,$$ and thus, by the inductive hypothesis, $v_i =  \sum_{k=1}^{m_i} p_{i,k}p_{i,k}^*.$ Therefore, $$\sum_{i=1}^l q_iq_i^* = \sum_{i=1}^{m-1} \sum_{j=1}^{n_i} e_{i,j}e_{i,j}^* + \sum_{i=m}^n \sum_{j=1}^{n_i} e_{i,j}v_ie_{i,j}^*= \sum_{i=1}^{m-1} \sum_{j=1}^{n_i} e_{i,j}e_{i,j}^* + \sum_{i=m}^n \sum_{j=1}^{n_i} e_{i,j}e_{i,j}^*,$$ which is $u$, by the (CK2) relation. (Since $E$ is acyclic, these are all the edges with source $u$.) Thus, $u = q_1q_1^*+ \dots + q_lq_l^*$ in $L_K(E)$.

(2) Let $F$ be the finite graph obtained from $E$ by setting $F^0 = D(u,m+1)$ and letting $F^1$ consist of all the edges $e \in E^1$ such that $s(e) \in D(u,m)$. Then, since $E$ has only regular vertices, the $q_i$ will be precisely all the paths in $F$ having source $u$ and range that is a sink. By (1), $u = q_1q_1^*+ \dots + q_lq_l^*$, as an expression in $L_K(F)$. It follows that the equality also holds as an expression in $L_K(E)$.
\end{proof}

We are now ready to give a sufficient condition (which will also turn out to be necessary) for a Leavitt path algebra to be a commutator ring.

\begin{proposition} \label{p_mult}
Let $K$ be a field having characteristic $p > 0$, and let $E$ be an acyclic graph having only regular vertices. Suppose that for each $u \in E^0$ there is an $m_u \in \N$ such that for all $w \in E^0$ satisfying $d(u,w) = m_u+1$, the number of paths $q=e_1\dots e_k \in \path(E)$ $(e_1, \dots, e_k \in E^1)$ such that $s(q) = u$, $r(q) = w$, and $s(e_2), \dots, s(e_k) \in D(u,m_u)$ is a multiple of $p$. Then $L_K(E) = [L_K(E), L_K(E)]$.
\end{proposition}

\begin{proof}
By Proposition~\ref{comm-ring}, it is enough to show that $u \in [L_K(E), L_K(E)]$, for each $u \in E^0$. Thus, let $u \in E^0$ be any vertex, and let $m_u \in \N$ be as in the statement. Let $v_1, \dots, v_n \in E^0$ be all the vertices such that $d(u,v_i) = m_u+1$. (We note that there is a finite number of such vertices, by our assumptions on $E$.) Also, for each $i \in \{1, \dots, n\}$ let $q_{i,1}, \dots, q_{i,n_i} \in \path(E)$ be all the paths $q=e_1\dots e_k \in \path(E)$ $(e_1, \dots, e_k \in E^1)$ such that $s(q) = u$, $r(q) = v_i$, and $s(e_2), \dots, s(e_k) \in D(u,m_u)$. Then $$\sum_{i=1}^n\sum_{j=1}^{n_i}[q_{i,j},q_{i,j}^*] =  \sum_{i=1}^n\sum_{j=1}^{n_i}q_{i,j}q_{i,j}^* - \sum_{i=1}^n\sum_{j=1}^{n_i}q_{i,j}^*q_{i,j} = \sum_{i=1}^n\sum_{j=1}^{n_i}q_{i,j}q_{i,j}^* - \sum_{i=1}^n n_iv_i = \sum_{i=1}^n\sum_{j=1}^{n_i}q_{i,j}q_{i,j}^*,$$ since, by hypothesis, each $n_i$ is a multiple of $p$. But, by Lemma~\ref{path_sum}(2), the resulting sum equals $u$, showing that $u \in [L_K(E), L_K(E)]$.
\end{proof}

It is easy to construct graphs $E$ satisfying the conditions of Proposition~\ref{p_mult}; below are three examples, where for each $u \in E^0$ we take $m_u = 0, 1$, and $2$, respectively. We also take $n_1, n _2, \dots$ to be an arbitrary sequence of positive integers, and use $(n)$ above an arrow to mean that there are $n$ edges connecting the relevant vertices.

$$
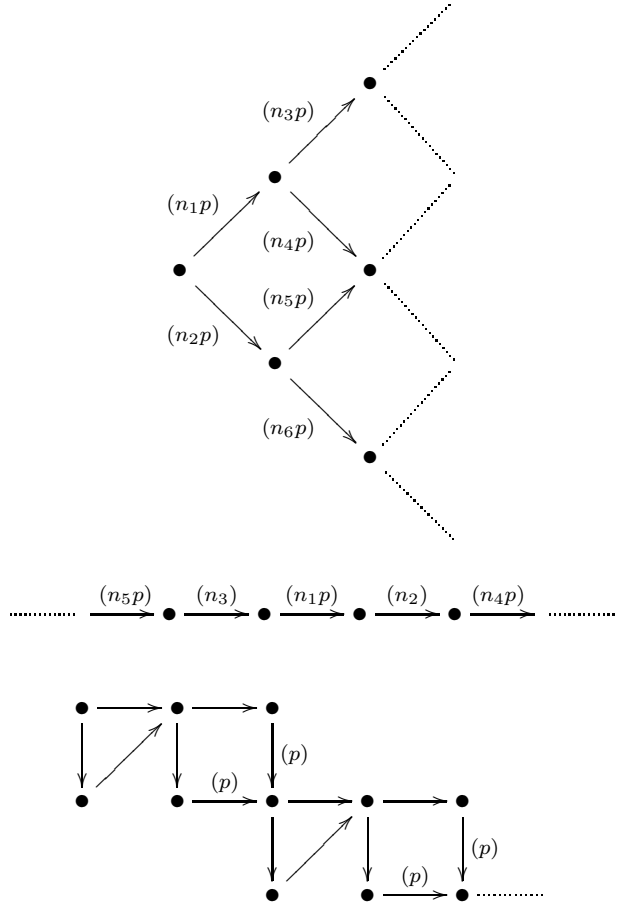
\xymatrix{
& & & \\ 
& & {\bullet} \ar@{.}[ur] \ar@{.}[dr] & \\
& {\bullet} \ar[ur]^{(n_3p)} \ar[dr]_{(n_4p)} & & \\
{\bullet} \ar[ur]^{(n_1p)} \ar[dr]_{(n_2p)} & & {\bullet} \ar@{.}[ur] \ar@{.}[dr]& \\
& {\bullet} \ar[ur]^{(n_5p)} \ar[dr]_{(n_6p)} & &\\
& & {\bullet} \ar@{.}[ur] \ar@{.}[dr] & \\
& & & \\ 
 }$$
 
$$\xymatrix{\ar@{.}[r] & \ar [r] ^{(n_5p)} & {\bullet} \ar [r] ^{(n_3)} & {\bullet} \ar [r] ^{(n_1p)} & {\bullet} \ar [r] ^{(n_2)} & {\bullet} \ar [r] ^{(n_4p)} & \ar@{.}[r] & }$$

$$\xymatrix{
{\bullet} \ar [r] \ar [d] & {\bullet} \ar [r] \ar [d] & {\bullet} \ar [d]^{(p)} & & & \\
{\bullet} \ar [ur] & {\bullet} \ar [r]^{(p)} & {\bullet} \ar [r] \ar [d]  & {\bullet} \ar [r]  \ar [d] & {\bullet} \ar [d]^{(p)} & \\
& & {\bullet} \ar [ur]  & {\bullet} \ar [r]^{(p)} & {\bullet} \ar@{.}[r] & \\
}$$

We note that the commutator rings resulting from the above class of graphs are fundamentally different from all the examples discussed in~\cite{ZM}. (All of the latter are unital, among other things.) Moreover, using the graphs above one can clearly obtain infinitely many non-isomorphic commutator Leavitt path algebras $L_K(E)$ by varying the characteristic of $K$. Similarly, one can construct non-isomorphic commutator Leavitt path algebras $L_K(E)$ by varying the cardinality of $E^0$. For instance, let $\kappa$ be an arbitrary infinite cardinal, and let $E$ be the graph pictured below, where $E^0 = \{v_i\}_{i=1}^\infty \cup \{u_i\}_{i \in \kappa}$.

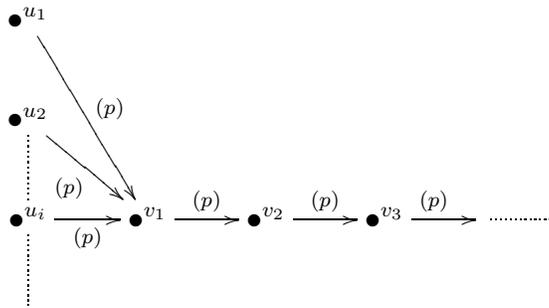
$$\xymatrix{
{\bullet}^{u_1} \ar[ddr]^{(p)} & & & & &\\
{\bullet}^{u_2} \ar[dr]_{(p)} \ar@{.}[d] & & & & &\\
{\bullet}^{u_i} \ar[r]_{(p)} \ar@{.}[d] & {\bullet}^{v_1} \ar [r] ^{(p)} & {\bullet}^{v_2} \ar [r] ^{(p)} & {\bullet}^{v_3} \ar [r] ^{(p)} & \ar@{.}[r] & \\
 & & & & &\\
}$$
Then $E^0$ has cardinality $\kappa$, and by Proposition~\ref{p_mult} (with $m_u = 0$ for each $u \in E^0$), $L_K(E)$ is a commutator ring, whenever $K$ is a field having characteristic $p$. 

Additionally, for a fixed field $K$ of characteristic $p > 0$ one can construct an infinite class of graphs $E_n$ ($n \in \Z_+$), all of the same cardinality, such that the resulting Leavitt path algebras $L_K(E_n)$ are pairwise non-isomorphic commutator rings. We shall give such a construction in the next section, in Proposition~\ref{perfectLPA}.

The rest of this section is devoted to showing that all the commutator Leavitt path algebras are of the form given in Proposition~\ref{p_mult}. We begin be recalling a result due to Abrams, Aranda Pino, and Siles Molina.

\begin{proposition}[Proposition 3.5 in \cite{AAS}] \label{fdLPA}
Let $K$ be a field, let $E$ be a finite acyclic graph, and let $v_1, \dots, v_t \in E^0$ be the sinks in $E$. Then $L_K(E)$ is isomorphic to a direct sum of full matrix rings over $K$. More specifically, $$L_K(E) \cong \bigoplus_{i=1}^t \M_{n(v_i)}(K),$$ where $n(v_i)$ is the cardinality of the set $\, \{p \in \path(E) \mid r(p) = v_i\}$.
\end{proposition}

We note, that in the proof of the above result, each sink $v_i \in E^0$ is mapped by the displayed isomorphism to a matrix in $\M_{n(v_i)}(K)$ having one $1$ somewhere on the main diagonal and zeros elsewhere. More generally, given a vertex $v \in E^0$, the projection of the image of $v$, under the isomorphism above, onto any direct summand is a matrix whose only nonzero entries are $1$s appearing somewhere on the main diagonal.

\begin{lemma} \label{sink-comm}
Let $K$ be a field, let $E$ be a finite acyclic graph, and let $u \in E^0$. If either $u$ is a sink or $K$ has characteristic $\, 0$, then $u \notin [L_K(E), L_K(E)]$.
\end{lemma}

\begin{proof}
It is well known that in a full matrix ring over a field only matrices having trace zero can be expressed as sums of commutators (e.g., see~\cite[Corollary 17]{ZM}). The result now follows from the above comment.
\end{proof}

\begin{lemma}\label{comm_ring_lemma1}
Suppose that $K$ is a field and $E$ is an acyclic graph such that $L_K(E)$ is a commutator ring. Then the following hold.
\begin{enumerate}
\item[$(1)$] The characteristic of $K$ is not $\, 0$.
\item[$(2)$] $E^0$ contains only regular vertices.
\item[$(3)$] $E^0$ is infinite.
\end{enumerate}
\end{lemma}

\begin{proof}
Let $F$ be the graph obtained from $E$ by removing all the edges whose source does not have finite out-degree. (So we turn each vertex with infinite out-degree in $E$ into a sink in $F$.) Then, by Definition~\ref{number-def}, $\{B_v \mid v \in E^0\} =  \{B_v \mid v \in F^0\}$. Hence, by Proposition~\ref{comm-ring}, $L_K(E)$ is a commutator ring if and only if $L_K(F)$ is a commutator ring.

Now, suppose that $u \in E^0$, and write $u = \sum_{i=1}^n [\mu_i,\nu_i] \in [L_K(E), L_K(E)]$ for some $\mu_i, \nu_i \in L_K(E)$. Since the terms of this sum involve only finitely many vertices and edges of $E$, there must be some finite subgraph $G$ of $E$, such that $u = \sum_{i=1}^n [\mu_i,\nu_i] \in [L_K(G), L_K(G)]$ for some $\mu_i, \nu_i \in L_K(G)$. Thus, by Lemma~\ref{sink-comm}, it cannot be the case that $u$ is a sink or that $K$ has characteristic $0$ (in particular, proving (1)). It follows that $E^0$ cannot contain any sinks, and hence, by the previous paragraph, $E^0$ cannot contain any vertices having infinite out-degree either, concluding the proof of (2).

Clearly no acyclic graph without sinks can be finite, showing that (3) holds.
\end{proof}

\begin{lemma} \label{comm_ring_lemma2}
Let $E$ be an acyclic graph having only regular vertices, $p$ a prime number, $K$ a field having characteristic $p$, and $u \in E^0$. Suppose that $u \in [L_K(E), L_K(E)]$. Then $\epsilon_u = \sum_{i=0}^n a_i B_{v_i} \ (\mathrm{mod} \ p)$ for some $a_i \in \Z$ and $v_i \in E^0$ which satisfy the following conditions:
\begin{enumerate}
\item[$(1)$] there is an $m \in \N$ such that $\, \{v_1, \dots, v_n\} = D(u,m);$
\item[$(2)$] $v_0 = u$ and $a_0 = -1;$
\item[$(3)$] for all $\, 0 \leq i < j \leq n$, we have $d(v_j,v_i) = \infty;$
\item[$(4)$] for all $\, 1 \leq i \leq n$, $-a_i$ is the number of paths $q=e_1\dots e_k \in \path(E)$ $(e_1, \dots, e_k \in E^1)$ such that $s(q) = u$, $r(q) = v_i,$ and $s(e_2), \dots, s(e_k) \in \{v_1, \dots, v_n\}.$
\end{enumerate}
Moreover, if $w \in E^0$ is a vertex such that $d(u,w) = m+1$, then the number of paths $q=e_1\dots e_k \in \path(E)$ $(e_1, \dots, e_k \in E^1)$ such that $s(q) = u$, $r(q) = w$, and $s(e_2), \dots, s(e_k) \in D(u,m)$ is a multiple of $p$.
\end{lemma}

\begin{proof}
By Theorem~\ref{LPA-comm} (or Theorem~\ref{All-Comm}), $u \in [L_K(E), L_K(E)]$ implies that $\epsilon_u = \sum_{i=0}^n a_i B_{v_i}$ for some $a_i \in K$ and $v_i \in E^0$. But, since the entries in each $B_{v_i}$ and in $\epsilon_u$ are integers, we may assume that each $a_i$ is in the prime subfield of $K$ (see~\cite[Remark 34]{AM}), or, equivalently, that each $a_i$ is an integer, and $\epsilon_u = \sum_{i=0}^n a_i B_{v_i}$ holds modulo $p$.

Suppose that for some $v_j$ we have $d(u, v_j) = \infty$. Upon replacing $v_j$ with a different vertex having the same property from $\{v_0, \dots, v_n\}$, if necessary, we may assume that for all $w \in \{u, v_0, \dots, v_n\} \setminus \{v_j\}$ we have $d(w, v_j) = \infty$ (since $E$ is acyclic). Let $\pi_{v_j} : \Z^{(E^0)} \rightarrow \Z$ be the natural projection onto the coordinate indexed by $v_j$. Then $\pi_{v_j}(\sum_{i=0}^n a_i B_{v_i}) = -a_j$, by the definition of $B_{v_i}$, which can only happen if $a_j = 0 \ (\mathrm{mod} \ p)$, since $\epsilon_u = \sum_{i=0}^n a_i B_{v_i}$ and $v_j \neq u$. Thus, we may remove any such $v_j$ and assume that for all $w \in \{v_0, \dots, v_n\}$ we have $d(u,w) \in \N$.

Letting $m\in \N$ be the largest value of $d(u, v_i)$ ($i \in \{0, \dots, n\}$), we see that $\{v_0, \dots, v_n\} \subseteq D(u,m)$. Since $E$ has only regular vertices, by adding vertices $v_i$ to the set $\{v_0, \dots, v_n\}$, if necessary,  and setting $a_i = p$, we may assume that $\{v_0, \dots, v_n\} = D(u,m)$. We note that since $E$ is acyclic, it must be the case that for all $v_j \in \{v_0, \dots, v_n\}\setminus \{u\}$ we have $d(v_j, u) = \infty$.

Next, let us reindex $v_1, \dots, v_n$ (and $a_0, \dots, a_n$ correspondingly) in any way so that $v_0 = u$ and condition (3) above is satisfied. This must be possible, since $E$ is acyclic, and since for all $v_j \in \{v_0, \dots, v_n\}\setminus \{u\}$ we have $d(v_j, u) = \infty$. It remains to show that $a_0 = -1$ and condition (4) is satisfied.

Since for each $0\leq i \leq n$, $B_{v_i}$ has only finitely many nonzero entries, upon projecting onto finitely many coordinates, we may identify $B_{v_i}$ with $(b_{i0}, \dots, b_{il})$, for some $l \in \N$ ($l \geq n$) and $b_{ij} \in \N$, where the first $n+1$ coordinates correspond to $v_0, \dots,  v_n$ when $B_{v_i}$ is viewed as an element of $\Z^{(E^0)}$. By the definition of $B_{v_i}$, we have $b_{ii} = -1$ for each $i$. Also, by condition (3), we have $b_{ij} = 0$ whenever $i > j$. Thus, \\ $B_{v_0} = (-1, b_{01}, b_{02},\dots\dots\dots\dots, b_{0l})$\\ $B_{v_1} = (\phantom{-}0,-1, b_{12}, \dots\dots\dots\dots, b_{1l})$\\ $\vdots$\\ $B_{v_n} = (\phantom{-}0, \dots, 0, -1, b_{n,n+1}, \dots, b_{nl})$.\\ In particular, we must have $a_0 = -1$. Let us now prove that (4) can be assumed to hold, by induction on $i$. We have $$0= \pi_{v_1}(\sum_{i=0}^n a_i B_{v_i}) = -1 \cdot a_1 + b_{01} \cdot a_0 = - a_1 - b_{01} \ (\mathrm{mod} \ p),$$ and hence replacing $a_i$ if necessary, we may assume that $-a_1 = b_{01}$ (as integers), which, by the definition of $B_{v_0}$, is the number of edges $e \in E^1$ such that $s(e) = v_0 = u$ and $r(e) = v_1$. But, by (1) and (3), the only paths from $u$ to $v_1$ that go through the vertices $v_1, \dots, v_n$ are edges, proving the claim for $a_1$. Let us therefore assume that (4) holds for $1 \leq j < n$, and prove it for $a_{j+1}$. Now $$0= \pi_{v_{j+1}}(\sum_{i=0}^n a_i B_{v_i}) = -1 \cdot a_{j+1} + \sum_{i=0}^j a_i b_{i,j+1} \ (\mathrm{mod} \ p),$$ and hence we may assume that $-a_{j+1} = \sum_{i=0}^j (-a_i) b_{i,j+1}$ (as integers). Since for each $0 \leq i \leq j$, $b_{i,j+1}$ is the number of edges $e \in E^1$ such that $s(e) = v_i$ and $r(e) = v_{j+1}$, the inductive hypothesis and (3) imply that $-a_{j+1}$ is the number of paths $q \in \path(E)$ that go only through $\{v_0, \dots, v_n\}$, such that $s(q) = u$ and $r(q) = v_{j+1}$, as desired.

For the final claim, let $w \in E^0$ is a vertex such that $d(u,w) = m+1$. Then there must be an edge $e \in E^1$ such that $r(e) = w$ and $s(e) = v_j$ for some $j \in \{0, \dots, n\}$. Hence, there must be some $n< k \leq l$ such that for all $i \in \{0, \dots, n\}$, $b_{ik}$ is the number of edges having source $v_i$ and range $w$ (where $B_{v_i} = (b_{i1}, \dots, b_{il})$). Letting $\pi_k : \Z^{(E^0)} \rightarrow \Z$ be the natural projection onto the coordinate corresponding to $k$, we have $$0= -\pi_k(\sum_{i=0}^n a_i B_{v_i}) = \sum_{i=0}^n (-a_i) b_{ik} \ (\mathrm{mod} \ p).$$ By (4), the last sum is the number of paths $q \in \path(E)$ that go only through $\{v_0, \dots, v_n\}$, such that $s(q) = u$ and $r(q) = w$. By (1), this sum is the desired number.
\end{proof}

We can now prove our main result about commutator Leavitt path algebras.

\begin{theorem} \label{PerfectLPAs}
Let $K$ be a field, and let $E$ be a graph. Then $L_K(E) = [L_K(E), L_K(E)]$ if and only if the following conditions hold.
\begin{enumerate}
\item[$(1)$] $E$ is acyclic.
\item[$(2)$] $E^0$ contains only regular vertices.
\item[$(3)$] The characteristic $p$ of $K$ is not $\, 0$.
\item[$(4)$] For each $u \in E^0$ there is an $m \in \N$ such that for all $w \in E^0$ satisfying $d(u,w) = m+1$, the number of paths $q=e_1\dots e_k \in \path(E)$ $(e_1, \dots, e_k \in E^1)$ such that $s(q) = u$, $r(q) = w$, and $s(e_2), \dots, s(e_k) \in D(u,m)$ is a multiple of $p$.
\end{enumerate}
\end{theorem}

\begin{proof}
By Proposition~\ref{p_mult}, if conditions (1) through (4) hold, then $L_K(E) = [L_K(E), L_K(E)]$. The converse follows from Proposition~\ref{comm-ring}, Lemma~\ref{comm_ring_lemma1}, and Lemma~\ref{comm_ring_lemma2}.  
\end{proof}

\section{Connections Between Ideals and Lie Ideals} \label{LieIdealSect}

The goal of this section is to show that the commutator Leavitt path algebras (described in Theorem~\ref{PerfectLPAs}) actually have the additional property that all their Lie ideals coincide with the ideals they generate. We begin with two lemmas.

\begin{lemma} \label{LieIdealLemma1}
Let $K$ be a field, let $E$ be an acyclic graph, let $\mu \in L_K(E)$ be an arbitrary nonzero element, and let $U \subseteq L_K(E)$ be the Lie ideal generated by $\mu$. Write $\mu = \sum_{i=1}^n a_ip_iq_i^*$ for some $a_i \in K \setminus \{0\}$ and $p_i, q_i \in \path(E)$ $($where $r(p_i)=r(q_i))$. Then the following hold.
\begin{enumerate}
\item[$(1)$] Let $l \in \{1, \dots, n\}$ be such that $q_l$ has maximal length among the $q_i$ satisfying $s(q_i) = s(q_l)$. Then for all $t \in \path(E) \setminus E^0$ such that $s(t) = r(p_l) = r(q_l)$ and $r(t) \neq s(p_i)$ for all $i \in  \{1, \dots, n\}$, we have $t \in U$.
\item[$(2)$] Let $l \in \{1, \dots, n\}$ be such that $p_l$ has maximal length among the $p_i$ satisfying $s(p_i) = s(p_l)$. Then for all $t \in \path(E) \setminus E^0$ such that $s(t) = r(p_l) = r(q_l)$ and $r(t) \neq s(q_i)$ for all $i \in  \{1, \dots, n\}$, we have $t^* \in U$.
\end{enumerate}
\end{lemma}

\begin{proof}
(1) Note that $$[\mu,q_lt] = \sum_{i=1}^n a_ip_iq_i^*(q_lt) - \sum_{i=1}^n a_i(q_lt)p_iq_i^* =\sum_{i=1}^n a_ip_iq_i^*q_lt - 0 =  \sum_{i=1}^m b_iu_it,$$ for some $b_i \in K \setminus \{0\}$ and distinct $u_i \in \path(E)$ satisfying $r(u_i)=s(t)$. It could not be the case that $r(u_i) = r(t)$ for some $i \in \{1, \dots, m\}$, since $E$ is acyclic and $t \notin E^0$. Thus, letting $k \in \{1, \dots, m\}$ be such that the length of $u_k$ is maximal, we have $$[u_k^*, \sum_{i=1}^m b_iu_it]= \sum_{i=1}^m b_iu_k^*u_it - \sum_{i=1}^m b_iu_itu_k^* = \sum_{i=1}^m b_iu_k^*u_it.$$ 

Since $r(u_i) = s(t) = r(u_k)$ for all $i \in \{1, \dots, m\}$, and since $E$ is acyclic, it cannot be the case that $u_k = u_iz$ for some $i \in \{1, \dots, m\}$ and $z \in \path(E) \setminus E^0$. Therefore, by the maximality of the length of $u_k$, if $i \neq k$, we have $u_k^*u_i = 0$. Hence $[u_k^*, [\mu,q_lt]] = b_kt$, and therefore $t \in U$.

(2) The proof is entirely analogous to that of (1), but we include the details for the convenience of the reader. We have $$[t^*p_l^*, \mu] = \sum_{i=1}^n a_i(t^*p_l^*)p_iq_i^* - \sum_{i=1}^n a_ip_iq_i^*(t^*p_l^*) =\sum_{i=1}^n a_it^*p_l^*p_iq_i^* - 0 =  \sum_{i=1}^m b_it^*u_i^*,$$ for some $b_i \in K \setminus \{0\}$ and distinct $u_i \in \path(E)$ satisfying $r(u_i)=s(t)$. As before, it could not be the case that $r(u_i) = r(t)$ for some $i \in \{1, \dots, m\}$, since $E$ is acyclic and $t \notin E^0$. Thus, letting $k \in \{1, \dots, m\}$ be such that the length of $u_k$ is maximal, we have $$[\sum_{i=1}^m b_it^*u_i^*,u_k]= \sum_{i=1}^m b_it^*u_i^*u_k - \sum_{i=1}^m b_iu_kt^*u_i^* = \sum_{i=1}^m b_it^*u_i^*u_k.$$ 

Since $r(u_i) = s(t) = r(u_k)$ for all $i \in \{1, \dots, m\}$, and since $E$ is acyclic, it cannot be the case that $u_k = u_iz$ for some $i \in \{1, \dots, m\}$ and $z \in \path(E) \setminus E^0$. Therefore, by the maximality of the length of $u_k$, if $i \neq k$, we have $u_i^*u_k = 0$. Hence $[[t^*p_l^*,\mu], u_k] = b_kt^*$, and therefore $t^* \in U$.
\end{proof}

\begin{lemma} \label{LieIdealLemma2}
Let $K$ be a field having characteristic $p > 0$, and let $E$ be an acyclic graph having only regular vertices. Suppose that for each $u \in E^0$ there is an $m_u \in \N$ such that for all $w \in E^0$ satisfying $d(u,w) = m_u+1$, the number of paths $q=e_1\dots e_k \in \path(E)$ $(e_1, \dots, e_k \in E^1)$ such that $s(q) = u$, $r(q) = w$, and $s(e_2), \dots, s(e_k) \in D(u,m_u)$ is a multiple of $p$. Then the following hold.
\begin{enumerate}
\item[$(1)$] Let $v,u \in E^0$, and let $U \subseteq L_K(E)$ be the Lie ideal generated by $v$. If there exists $t \in \path(E) \setminus E^0$ such that $s(t) = v$ and $r(t) = u$, then $u \in U$.
\item[$(2)$] Let $v \in E^0$, and let $U \subseteq L_K(E)$ be the Lie ideal generated by $v$. Then for all $t, q \in \path(E)$ such that $s(t) = v$, $s(q) = v$, or $r(t) = v = r(q)$, we have $tq^* \in U$.
\item[$(3)$] Let $\mu = \sum_{i=1}^n a_it_iq_i^*\in L_K(E) \setminus \{0\}$ be an arbitrary element $($where $a_i \in K \setminus \{0\}$ and $t_i, q_i \in \path(E)$$)$, and let $U \subseteq L_K(E)$ be the Lie ideal generated by $\mu$. Then $r(t_i) \in U$ for all $i \in \{1, \dots, n\}$. 
\end{enumerate}
\end{lemma}

\begin{proof}
(1) We have $[t^*,v] = t^*v-vt^* = t^* \in U$ ($vt^* = 0$, since $E$ is acyclic). Now, let  $m_u \in \N$ be as in the statement, and let $q_1, \dots, q_n \in \path(E)$ be all the paths through $D(u,m_u)$ such that $d(u,r(q_i)) = m_u+1$ and $s(q_i)=u$. Then for all $i \in \{1, \dots, n\}$ we have $[t^*, tq_i] = q_i - tq_it^* = q_i \in U$ (again $q_it^* = 0$, since $E$ is acyclic, and so $r(t) = s(q_i) \neq r(q_i)$). Thus, by Lemma~\ref{path_sum}(2), $$\sum_{i=1}^n [q_i,q_i^*] = \sum_{i=1}^n q_iq_i^* - \sum_{i=1}^n q_i^*q_i = u - \sum_{i=1}^n q_i^*q_i\in U.$$ By hypothesis, the final sum in the above display consists of vertices, each appearing some number of times that is divisible by $p$, and therefore this implies that $u \in U$.

(2) We may assume that $r(t) = r(q)$, since otherwise $tq^*=0$ and there is nothing to prove.

If $s(t) = v$ and $t \notin E^0$, then by Lemma~\ref{LieIdealLemma1}(1), $t \in U$. Hence $[t,q^*] = tq^* - q^*t \in U$. Now, suppose that $q^*t \neq 0$. Then $s(q) = s(t)$, and since $r(q) = r(t)$ and $E$ is acyclic, this implies that $q = t$, and hence $q^*t = r(t)$. Thus either $[t,q^*] = tq^*$ or $[t,q^*] = tq^* - r(t)$. By (1), it follows that $tq^* \in U$ in either case.

If $s(q) = v$ and $q \notin E^0$, then by Lemma~\ref{LieIdealLemma1}(2), $q^* \in U$. Hence $[t,q^*] = tq^* - q^*t \in U$. Again $q^*t$ is either $0$ or $r(q)$, and in either case we have $tq^* \in U$.

Next, suppose that $r(t) = v = r(q)$ and $t \notin E^0$. Then $[t,v] = t \in U$, and hence $[t,q^*] = tq^* - q^*t \in U$. As before, this implies that $tq^* \in U$ (since $r(t) = v \in U$).

If $r(t) = v = r(q)$ and $q \notin E^0$. Then $[v,q^*] = q^* \in U$, and hence $[t,q^*] = tq^* - q^*t \in U$. As usual, it follows that $tq^* \in U$.

Finally, if $t=v=q$, then $tq^* = v \in U$.

(3) Let $l \in \{1, \dots, n\}$ be such that $q_l$ has maximal length among the $q_i$ having the same source, and let $m_{r(t_l)} \in \N$ be as in the statement of this lemma. By increasing $m_{r(t_l)}$, if necessary, we may assume that for all paths $z \in \path(E)$ through $D(r(t_l),m_{r(t_l)})$ such that $d(r(t_l),r(z)) = m_{r(t_l)}+1$ and $s(z)=r(t_l)$, we have $r(z) \neq s(t_i)$ for all $i \in  \{1, \dots, n\}$. (It is easy to see, by induction, for every $u \in E^0$ and $m > m_u$, that $m$ also satisfies the hypotheses of the lemma, in place of $m_u$.) Let $z_1, \dots, z_k \in \path(E)$ be all the paths $z_i\in \path(E)$ through $D(r(t_l),m_{r(t_l)})$ such that $d(r(t_l),r(z_i)) = m_{r(t_l)}+1$ and $s(z_i)=r(t_l)$. Then, by Lemma~\ref{LieIdealLemma1}(1), for each $i \in  \{1, \dots, k\}$, we have $z_i \in U$. Therefore, by Lemma~\ref{path_sum}(2), $$\sum_{i=1}^k [z_i,z_i^*] = \sum_{i=1}^k z_iz_i^* - \sum_{i=1}^k z_i^*z_i = r(t_l) - \sum_{i=1}^k z_i^*z_i\in U.$$ Again, the final sum in the above display consists of vertices, each appearing some number of times that is divisible by $p$, and therefore this implies that $r(t_l) \in U$.

Now, by (2), we have $t_lq_l^* \in U$, and hence $\mu - t_lq_l^* \in U$, from which one can conclude that $r(t_i) \in U$ whenever $s(q_i) = s(q_l)$ (by repeating the entire procedure above on $\mu - t_lq_l^*$, and so on). It follows that $r(t_i) \in U$ for all $i \in \{1, \dots, n\}$.
\end{proof}

We are now ready for the main result of this section.

\begin{theorem} \label{ideal_iff_Lie_ideal}
Let $K$ be a field having characteristic $p > 0$, and let $E$ be an acyclic graph having only regular vertices. Suppose that for each $u \in E^0$ there is an $m_u \in \N$ such that for all $w \in E^0$ satisfying $d(u,w) = m_u+1$, the number of paths $q=e_1\dots e_k \in \path(E)$ $(e_1, \dots, e_k \in E^1)$ such that $s(q) = u$, $r(q) = w$, and $s(e_2), \dots, s(e_k) \in D(u,m_u)$ is a multiple of $p$.

Then a subset $U \subseteq L_K(E)$ is an ideal if and only if it is a Lie ideal. Moreover, if $U \subseteq L_K(E)$ is an ideal, then $[U, L_K(E)] = U$.
\end{theorem}

\begin{proof}
In any ring, every ideal is a Lie ideal, and so we need only prove that if $U$ is a Lie ideal, then it is also an ideal. Furthermore, since in this case $U$ is a $K$-subspace of $L_K(E)$, we need only show that for all $\mu \in U \setminus \{0\}$ and $w,u \in \path(E)$ we have $\mu wu^*, wu^*\mu \in U$. (Elements of the form $wu^*$ generate $L_K(E)$ as a $K$-vector space.) But, $[wu^*, \mu] = wu^*\mu - \mu wu^* \in U$, and hence $\mu wu^* \in U$ if and only if $wu^*\mu \in U$. Thus, it suffices to show that $\mu wu^* \in U$.

Write $\mu = \sum_{i=1}^n a_it_iq_i^*$ for some $a_i \in K \setminus \{0\}$ and $t_i, q_i \in \path(E)$. By Lemma~\ref{LieIdealLemma2}(2,3), $t_iq_i^* \in U$ for all $i \in \{1, \dots, n\}$, and hence we may assume that $\mu = tq^*$ for some $t, q \in \path(E)$. Now, if $\mu wu^* = 0$, then we are done. So let us assume that $\mu wu^* \neq 0$. This can happen only if there is some $z \in \path(E)$ such that either $q = wz$ or $w = qz$. Thus, either $tq^*wu^* = tz^*u^*$ or $tq^*wu^*= tzu^*$. Now, by Lemma~\ref{LieIdealLemma2}(3), $r(t) \in U$, and by part (2) of the same lemma, $tz^*u^* \in U$. Let us therefore suppose that $tq^*wu^*= tzu^*$ and that $z \notin E^0$ (for otherwise $tq^*wu^* = tz^*u^*$). By Lemma~\ref{LieIdealLemma2}(2), we still have $t \in U$. If $t \in E^0$, then the same lemma implies that $tzu^* = zu^* \in U$. So let us also suppose that $t \notin E^0$. Then $[t, z] = tz - zt = tz - 0 = tz\in U$ ($zt=0$, since $E$ is acyclic). Finally, this implies that $tzu^* \in U$, by Lemma~\ref{LieIdealLemma2}(2,3). Thus, in all cases, $\mu wu^* \in U$, as desired.

For the final claim, let $\mu = \sum_{i=1}^n a_it_iq_i^* \in U \setminus \{0\}$ (for some $a_i \in K \setminus \{0\}$ and $t_i, q_i \in \path(E)$). We wish to show that $\mu \in [U,L_K(E)]$. Again, by Lemma~\ref{LieIdealLemma2}(2,3), $t_iq_i^* \in U$ for all $i \in \{1, \dots, n\}$, and hence we may assume that $\mu = tq^*$ for some $t, q \in \path(E)$ satisfying $r(t)=r(q)$. By the same lemma, we have $t \in U$. If $t \neq q$, since $E$ is acyclic, we have $tq^* = tq^*-q^*t = [t,q^*] \in [U,L_K(E)]$. Let us therefore suppose that $t=q$, let $m_{r(t)} \in \N$ be as in the statement, and let $z_1,\dots, z_n \in \path(E)$ be all the paths $z_i\in \path(E)$ through $D(r(t),m_{r(t)})$ such that $d(r(t),r(z_i)) = m_{r(t)}+1$ and $s(z_i)=r(t)$. Then, by Lemma~\ref{path_sum}(2) and our hypotheses, using the usual argument, we have $$tq^* = tt^* - r(t) + r(t) = [t,t^*] + r(t) = [t,t^*] + \sum_{i=1}^n z_iz_i^* = [t,t^*] + \sum_{i=1}^n [z_i, z_i^*].$$ Now, each $z_i \in U$, by Lemma~\ref{LieIdealLemma2}(2), and hence the above display implies that $tq^* \in [U,L_K(E)]$.
\end{proof}

Not all commutator rings have the property that the ideals and Lie ideals coincide. For instance, if $R$ is a commutator ring with nonzero center, then any additive subgroup of the center is a Lie ideal, but certainly is not always an ideal. Less trivially, let $K$ be a field and let $R = \End_R(\bigoplus_{i=1}^\infty K)$ be the ring of row-finite matrices over $K$. Then $R = [R,R]$, by~\cite[Theorem 13]{ZM}. It is well known that $R$ has exactly three ideals, namely $0$, $R$, and the ideal $I$ consisting of the endomorphisms having finite rank, i.e.\ the matrices having nonzero entries in only finitely many columns (e.g., see~\cite[Exercises 3.14-3.16]{TL}). It is easy to show that the Lie ideal $[I, R]$ consists of the elements of $I$ that have trace zero (each matrix in $I$ has only finitely many nonzero entries on the main diagonal, and so it makes sense to compute its trace), and hence is not an ideal of $R$.

The rest of the section is devoted to building examples commutator Leavitt path algebras with various (Lie) ideal structures. We first recall some known results and a previously-defined notion.

\begin{definition}
A graph $E$ is said to satisfy condition (K) if for each vertex $v \in E^0$ such that there is a \emph{simple closed path based at} $v$ (i.e., there is a path $p=e_1\dots e_n \in \path(E) \setminus E^0$, for some $e_1, \dots, e_n \in E^1$, such that $s(p)=v=r(p)$ and $s(e_i) \neq v$ for $i >1$), there are at least two distinct closed simple paths based at $v$. \hfill $\Box$
\end{definition}

The following statement is a combination of a theorem due to Ara, Moreno, and Pardo with a theorem of Aranda Pino, Pardo, and Siles Molina. (See Section~\ref{NotationSect} for the definitions of \emph{hereditary} and \emph{saturated}.)

\begin{theorem}[Theorem 5.3 in \cite{AMP} and Theorem 4.5(3,4) in \cite{APM2}] \label{AMPAPMTheorem}
Let $K$ be a field, and let $E$ be a row-finite graph. We let $\mathcal{H}$ denote the lattice of all subsets of $E^0$ that are hereditary and saturated, let $\mathcal{L}(L_K(E))$ denote the lattice of all ideals of $L_K(E)$, and let $\mathcal{L}_{gr}(L_K(E))$ denote the lattice of all graded ideals of $L_K(E)$. Then the following hold.
\begin{enumerate}
\item[$(1)$] There is an order-isomorphism between $\mathcal{H}$ and $\mathcal{L}_{gr}(L_K(E))$.
\item[$(2)$] $\mathcal{L}(L_K(E)) = \mathcal{L}_{gr}(L_K(E))$ if and only if $E$ satisfies condition $($K$)$.
\end{enumerate}
\end{theorem}

\begin{corollary} \label{AMPAPMCor}
Let $K$ be a field, and let $E$ be an acyclic row-finite graph. Then there is an order isomorphism between the lattice of all hereditary saturated subsets of $E^0$ and the lattice of all ideals of $L_K(E)$.

In particular, there is such an order isomorphism whenever $L_K(E) = [L_K(E), L_K(E)]$.
\end{corollary}

\begin{proof}
This follows from Theorem~\ref{AMPAPMTheorem}, since an acyclic graph $E$ satisfies condition (K) vacuously. The last claim follows from Theorem~\ref{PerfectLPAs}, which implies that if $L_K(E) = [L_K(E), L_K(E)]$, then $E$ must be acyclic and row-finite.
\end{proof}

With the help of the above corollary we can now construct an infinite class of pairwise non-isomorphic commutator Leavitt path algebras, where the field is fixed, and all the graphs have the same cardinality (as promised in the previous section).

\begin{proposition} \label{perfectLPA}
Let $K$ be a field having characteristic $p > 0$, and for each $n \in \Z_+$ let $E_n$ be the graph pictured below. $$
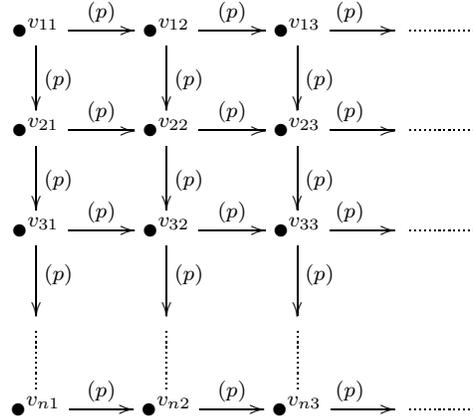
\xymatrix{ 
{\bullet}^{v_{11}} \ar [r] ^{(p)} \ar[d]^{(p)} & {\bullet}^{v_{12}} \ar [r] ^{(p)} \ar[d]^{(p)} & {\bullet}^{v_{13}} \ar [r] ^{(p)}  \ar[d]^{(p)} & \ar@{.}[r] & \\
{\bullet}^{v_{21}} \ar [r] ^{(p)} \ar[d]^{(p)} & {\bullet}^{v_{22}} \ar [r] ^{(p)} \ar[d]^{(p)} & {\bullet}^{v_{23}} \ar [r] ^{(p)} \ar[d]^{(p)} & \ar@{.}[r] &\\
{\bullet}^{v_{31}} \ar [r] ^{(p)} \ar[d]^{(p)} & {\bullet}^{v_{32}} \ar [r] ^{(p)} \ar[d]^{(p)}  & {\bullet}^{v_{33}} \ar [r] ^{(p)}\ar[d]^{(p)} & \ar@{.}[r] &\\
\ar@{.}[d] & \ar@{.}[d] & \ar@{.}[d] & & \\
{\bullet}^{v_{n1}} \ar [r] ^{(p)} & {\bullet}^{v_{n2}} \ar [r] ^{(p)} & {\bullet}^{v_{n3}} \ar [r] ^{(p)} & \ar@{.}[r] &
}$$ Then for each $n \in \Z_+$, $L_K(E_n) = [L_K(E_n), L_K(E_n)]$, and $L_K(E)$ has precisely $n+1$ ideals, which form a chain. In particular, $L_K(E_n) \not\cong L_K(E_m)$ if $n \neq m$.
\end{proposition}

\begin{proof}
By Proposition~\ref{p_mult} (with $m_u=0$ for all $u \in E^0$), we have $L_K(E_n) = [L_K(E_n), L_K(E_n)]$. We note that  for each $k \in \{1, \dots, n\}$, the set $\bigcup_{i=k}^n(\bigcup_{j=1}^\infty \{v_{ij}\}) \subseteq E^0$ is hereditary and saturated. On the other hand, suppose that $H \subseteq E_n^0$ is a nonempty hereditary saturated subset, and let $v_{kl} \in H$ be a vertex with $k$ minimal.  Since $H$ is hereditary, we must have $v_{nl} \in H$. Since $H$ is saturated, we must then have $v_{n1}, \dots, v_{n,l-1} \in H$, and hence also $v_{ij} \in H$ for $i \geq k$ and $j \leq l$. Using the assumption that $H$ is hereditary again, it follows that $H = \bigcup_{i=k}^n(\bigcup_{j=1}^\infty \{v_{ij}\})$. Thus, these, together with the empty set, are precisely the hereditary saturated subsets of $E^0$. Since these subsets form a chain, the desired conclusion follows from Corollary~\ref{AMPAPMCor}.

The final claim is immediate.
\end{proof}

Using a similar approach, one can construct commutator Leavitt path algebras with infinitely many ideals.

\begin{example}
Let $K$ be a field having characteristic $p > 0$, and let $E_\infty$ be the graph pictured below. $$
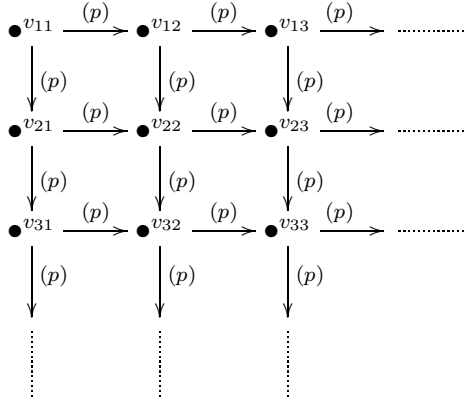
\xymatrix{ 
{\bullet}^{v_{11}} \ar [r] ^{(p)} \ar[d]^{(p)} & {\bullet}^{v_{12}} \ar [r] ^{(p)} \ar[d]^{(p)} & {\bullet}^{v_{13}} \ar [r] ^{(p)}  \ar[d]^{(p)} & \ar@{.}[r] & \\
{\bullet}^{v_{21}} \ar [r] ^{(p)} \ar[d]^{(p)} & {\bullet}^{v_{22}} \ar [r] ^{(p)} \ar[d]^{(p)} & {\bullet}^{v_{23}} \ar [r] ^{(p)} \ar[d]^{(p)} & \ar@{.}[r] &\\
{\bullet}^{v_{31}} \ar [r] ^{(p)} \ar[d]^{(p)} & {\bullet}^{v_{32}} \ar [r] ^{(p)} \ar[d]^{(p)}  & {\bullet}^{v_{33}} \ar [r] ^{(p)}\ar[d]^{(p)} & \ar@{.}[r] &\\
\ar@{.}[d] & \ar@{.}[d] & \ar@{.}[d] & & \\
&&&&
}$$ By the same argument as in Proposition~\ref{perfectLPA}, $L_K(E_\infty) = [L_K(E_\infty), L_K(E_\infty)]$. It is also easy to see that the hereditary saturated subsets of $E_\infty^0$ are of the form $\{v_{ij} \mid i \geq k, j \geq l\}$ ($k,l \in \Z_+$). Thus, by Corollary~\ref{AMPAPMCor}, $L_K(E_\infty)$ has infinitely many ideals, and they do not form a chain.  \hfill $\Box$
\end{example}

We mention in passing that the rings $L_K(E_n)$ and $L_K(E_\infty)$ constructed above are prime, by \cite[Proposition 5.6]{APM2}. 


\section{Direct Limits}

We conclude by briefly noting that the commutator rings produced in the previous two sections are actually special cases of a more general construction described below, though it is not always informative to view the commutator Leavitt path algebras as such.

\begin{lemma} \label{DirLimLemma}
Let $\,(I, \leq)$ be a directed set, let $\,\{R_i\}_{i\in I}$ be a family of $\,($not necessarily unital$)$ rings, and let $\,\{f_{ij} \mid  i,j \in I, \ i \leq j \}$ be a family of injective ring homomorphisms $f_{ij} : R_i \rightarrow R_j$, such that $((R_i)_{i\in I}, (f_{ij})_{i\leq j})$ forms a directed system. Also, let $D$ be the direct limit of this system. If for each $i \in I$ there is a $j \in I$ with $i < j$, such that $f_{ij}(R_i) \subseteq [R_j, R_j]$, then $D = [D,D] \neq 0$.
\end{lemma}

\begin{proof}
Let $x \in D$ be any element. Then $x \in R_i$ for some $i \in I$. Let $j \in I$ be such that $i < j$ and $f_{ij}(x) \in [R_j, R_j]$. Then $x = f_{ij}(x)$ as elements of $D$, and hence $x \in [D,D]$. That $D \neq 0$ follows from the injectivity of the transition maps $f_{ij}$.
\end{proof}

\begin{corollary} \label{DirLimMatrix}
Let $\, 1<n\leq m$ be integers, and let $R$ be a unital ring having characteristic $n$. For each $i \in \N$, let $f_i : \M_{m^i}(R) \rightarrow \M_{m^{i+1}}(R)$ be the map that inflates each entry $r \in R$ to the corresponding $m \times m$ matrix that has $r$ as each of the first $n$ entries on the main diagonal and zeros elsewhere. Let $D$ be the direct limit of the resulting directed system. Then $D = [D,D]$.
\end{corollary}

\begin{proof}
It is easy to verify that each $f_i$ is an injective ring homomorphism. Also, for any $i$ and $A \in \M_{m^i}(R)$, $f_i(A) \in \M_{m^{i+1}}(R)$ is a matrix having trace $0$ (since this trace is a multiple of $n$). Hence $f_i(A) \in [\M_{m^{i+1}}(R), \M_{m^{i+1}}(R)]$ (e.g., by~\cite[Corollary 17]{ZM}). It now follows from Lemma~\ref{DirLimLemma} that $D = [D,D]$.
\end{proof}

Some of the commutator Leavitt path algebras constructed above are of the form $D$ in Corollary~\ref{DirLimMatrix}. For instance, let $E_1$ be as in Proposition~\ref{perfectLPA}. Then taking $R=K$, $n=p$, and $m=p+1$, and letting $D$ be the corresponding direct limit, it is not hard to show that $L_K(E) \cong D$.

Others among our commutator Leavitt path algebras are more difficult to recognize as direct limits of the sort presented in Lemma~\ref{DirLimLemma}, but this can be accomplished using a more complicated version of the construction in Corollary~\ref{DirLimMatrix}. We leave the details to the interested reader.

\vspace{.1in}

\noindent
Department of Mathematics \newline
University of Colorado \newline
Colorado Springs, CO 80918 \newline
USA \newline

\noindent {\tt zmesyan@uccs.edu}


\begin{thebibliography}{00}
\bibitem{AP} G.\ Abrams and G.\ Aranda Pino, \textit{The Leavitt path algebra of a graph,} J.\ Algebra \textbf{293} (2005) 319--334.

\bibitem{AAS} G.\ Abrams, G.\ Aranda Pino, and M.\ Siles Molina, 
\textit{Finite-dimensional Leavitt path algebras,} J.\ Pure Appl.\ Algebra \textbf{209} (2007) 753--762.

\bibitem{AF} G.\ Abrams and D.\ Funk-Neubauer, \textit{On the simplicity of Lie algebras associated to Leavitt algebras,} Comm.\ Algebra \textbf{39} (2011) 1--11.

\bibitem{AM} G.\ Abrams and Z.\ Mesyan, \textit{Simple Lie algebras arising from Leavitt path algebras,} J.\ Pure Appl.\ Algebra, to appear.

\bibitem{AABM} A.\ A.\ Albert and B.\ Muckenhoupt,
\textit{On matrices of trace zero,} Michigan Math.\ J.\ \textbf{4}
(1957) 1--3.

\bibitem{AMP} P.\ Ara, M.\ A.\ Moreno, and E.\ Pardo, \textit{Nonstable K-theory for graph algebras,} Algebr.\ Represent.\ Theory \textbf{10} (2007) 157--178.

\bibitem{APM2} G.\ Aranda Pino, E.\ Pardo, and M.\ Siles Molina, \textit{Exchange Leavitt path algebras and stable rank,} J.\ Algebra \textbf{305} (2006) 912--936.


\bibitem{BZ} G.\ Benkart and E.\ Zelmanov, \textit{Lie algebras graded by finite root systems and intersection matrix algebras,} Invent.\ Math.\ \textbf{126} (1996) 1--45.

\bibitem{BHK} J.\ Bergen, I.\ N.\ Herstein, and J.\ W.\ Kerr, \textit{Lie ideals and derivations of prime rings,} J.\ Algebra \textbf{71} (1981) 259--267. 

\bibitem{Erickson} T.\ Erickson, \textit{The Lie structure in prime rings with involution,} J.\ Algebra \textbf{21} (1972) 523--534. 

\bibitem{Goodearl} K.\ R.\ Goodearl, \textit{Leavitt path algebras and direct limits,} in Rings, Modules, and Representations (N.\ V.\ Dung, et al., eds.), Contemp.\ Math.\ \textbf{480} (2009) 165--187. 

\bibitem{Harris} B.\ Harris, \textit{Commutators in division rings,} Proc.\ Amer.\ Math.\ Soc.\ \textbf{9} (1958) 628--630.

\bibitem{Herstein1} I.\ N.\ Herstein, \textit{Lie and Jordan structures in simple, associative rings,} Bull.\ Amer.\ Math.\ Soc.\ \textbf{67} (1961) 517--531.

\bibitem{Herstein2} I.\ N.\ Herstein, \textit{On the Lie structure of an associative ring,} J.\ Algebra \textbf{14} (1970) 561--571. 

\bibitem{Jacobson} N.\ Jacobson, \textit{Lie algebras,} Wiley-Interscience, New York, 1962.

\bibitem{Kaplansky} I.\ Kaplansky, \textit{``Problems in the theory of rings" revisited,} Amer.\ Math.\ Monthly \textbf{77} (1970) 445--454.


\bibitem{TL} T.\ Y.\ Lam, \textit{A first course in noncommutative rings,} Second Edition, Graduate Texts in Math., Vol.\ \textbf{131}, Springer-Verlag, Berlin-Heidelberg-New York, 2001.

\bibitem{LM} C.\ Lanski and S.\ Montgomery, \textit{Lie structure of prime rings of characteristic 2,} Pacific J.\ Math. \textbf{42} (1972) 117--136. 

\bibitem{Leavitt} W.\ G.\ Leavitt, \textit{The module type of a ring,} Trans.\ Amer.\ Math.\ Soc.\ \textbf{42} (1962) 113--130.

\bibitem{Marshall} E.\ I.\ Marshall, \textit{The genus of a perfect Lie algebra,} J.\ London Math.\ Soc.\ \textbf{40} (1965) 276--282. 

\bibitem{ZM} Z.\ Mesyan, \textit{Commutator rings,} Bull.\ Austral.\ Math Soc.\ \textbf{74} (2006) 279--288.

\bibitem{Raeburn} I.\ Raeburn, \textit{Graph algebras,} CBMS Regional Conference Series in Mathematics, \textbf{103},
Published for the Conference Board of the Mathematical Sciences, Washington, DC; by the American Mathematical Society, Providence, RI, 2005.

\bibitem{Shoda} K.\ Shoda, \textit{Einige S\"{a}tze \"{u}ber Matrizen,} Japan J.\ Math \textbf{13} (1936) 361--365.

\end{thebibliography}
\end{document}